\documentclass[12pt]{amsart}
\def \R{\mathbb{R}}

\def\co{\colon}

\usepackage[normalem]{ulem}
\usepackage[mathscr]{euscript}

\DeclareSymbolFont{rsfs}{U}{rsfs}{m}{n}
\DeclareSymbolFontAlphabet{\mathscrsfs}{rsfs}

\usepackage{tikz-cd}

\usepackage{graphicx, amscd}
\usepackage{mathabx}
\newtheorem{theorem}{Theorem}
\newtheorem{lemma}[theorem]{Lemma}
\newtheorem{corollary}[theorem]{Corollary}
\newtheorem{conjecture}[theorem]{Conjecture}
\newtheorem{proposition}[theorem]{Proposition}

\theoremstyle{definition}

\newtheorem{remark}[theorem]{Remark}
\newtheorem{example}[theorem]{Example}

\usepackage{geometry}

\title{Isolated singularities of hypersurfaces}
\date{\today}

\begin{document}
\author{Rustam Sadykov}
\author{Stanislav Trunov}
\begin{abstract} 

 Introduced by Seifert and Threlfall, cylindrical neighborhoods is an essential tool in the Lusternik-Schnirelmann theory. We conjecture that every isolated critical point of a smooth function admits a cylindrical ball neighborhood. We show that the conjecture is true for cone-like critical points, Cornea reasonable critical points, and critical points that satisfy the Rothe $H$ hypothesis. In particular, the conjecture holds true at least for those critical points that are not infinitely degenerate. 
\end{abstract}

\maketitle

\section{Introduction}

Given a smooth function $f\co M\to \R$ on a manifold $M$ without boundary, a point $x$ in $M$ is said to be \emph{critical} if the differential $d_xf$ of $f$ at $x$ is trivial. A point that is not critical is said to be \emph{regular}. 
We say that a value $c$ of $f$ is \emph{critical} if the fiber $f^{-1}(c)$ contains a critical point.  Critical points of functions could be extremely complicated. For example, the fiber $f^{-1}(c)$ over a regular value $c$ is a hypersurface of $M$ with no singularities. On the other hand, every closed subset $V\subset M$ is the hypersurface level $f^{-1}(c)$ of an appropriately chosen smooth function $f$ on $M$. In other words, critical points of smooth functions are at least as complicated as closed subsets of a manifold. 

Let $x_0$ be a critical point of a smooth function $f$ with critical value $f(x_0)=c$. 
We say that an isolating neighborhood $U$ of $x_0$  is \emph{cylindrical} if it consists of trajectories in $f^{-1}[c-\varepsilon, c+\varepsilon]$ of the gradient vector field of $f$. We will always assume that $U$ is a smooth manifold of dimension $\dim M$ with corners such that $U\cap f^{-1}(c-\varepsilon)$ and $U\cap f^{-1}(c+\varepsilon)$ are smooth manifolds with (smooth) boundary. Cylindrical neighborhoods were introduced by Seifert and Threlfall \cite{ST38} in 1938, and were used in \cite{Ro50, Ta68, Da84} and \cite{Co98}. The existence of a cylindrical ball neighborhood, i.e., a cylindrical neighborhood homeomorphic to a ball, is an important property as it allows one to apply, for example, the Lusternik-Schnirelman type argument  and deduce Lusternik-Schnirelman inequalities \cite{Ta68, Co98, ST}.

In this paper we study cylindrical ball neighborhoods and give evidence supporting Conjecture~\ref{c:main_cyl}. We note that in dimensions $\ne 4,5$ all smooth manifolds homeomorphic to a ball are diffeomorphic to a ball. 

\begin{conjecture}\label{c:main_cyl}
Every isolated critical point of a smooth function admits a cylindrical ball neighborhood. 
\end{conjecture}

If there is an open subset $U\subset \R^m$, and a coordinate neighborhood $\tau\co U\to M$ of a point $x\in M$  such that $f\circ \tau$ is the restriction of a polynomial, we say that the critical point $x$ of $f$ is \emph{algebraic}. In this case, the singular hypersurface $V=f^{-1}(c)$ is locally an algebraic set. It is known (e.g. see \cite{Mi69}) that for every algebraic critical point $x_0$ of $f$ with $f(x_0)=c$, there is a closed disc neighborhood $B_\varepsilon$ of $x_0$ in $M$ such that $\partial B_\varepsilon$ is transverse to $V$, and  the pair $(B_\varepsilon, V\cap B_\varepsilon)$ is homeomorphic to the cone over the pair $(\partial B_\varepsilon, V\cap \partial B_\varepsilon)$. We say that $(B_\varepsilon, V\cap  B_\varepsilon)$ is a \emph{cone neighborhood pair} for $x_0$.  
 This essential property of algebraic singular points is a starting point for the theory of hypersurface singularities, see \cite{Mi69}.   
It follows (see \cite{Ki78, Ki90} and Proposition~\ref{prop:cone16}) that a critical point $x_0$ admits a cone neighborhood pair if and only if it is \emph{cone-like}, i.e., 
if it admits a cone neighborhood in $f^{-1}(f(x_0))$.

\begin{theorem}\label{th:1} Every cone-like critical point of a smooth function $f\co M\to \R$ admits a cylindrical ball neighborhood. \end{theorem}

A version of Theorem~\ref{th:1} appears in \cite[p. 396]{Ki78} for manifolds $M$ of dimension $m\ne 4$.

In \cite{Co98}, Cornea defined a \emph{reasonable critical point} as an isolated critical point $x_0$ of $f$ such that $f^{-1}(f(x_0))$ admits a Whitney stratification into the stratum $\{x_0\}$ and its complement. By \cite[Lemma 2]{Co98} every reasonable critical point admits a cone neighborhood pair, and, therefore,  it
is cone-like. 

\begin{corollary} Every reasonable critical point admits a cylindrical ball neighborhood.
\end{corollary}

Another large class of critical points was studied by E.~Rothe in \cite{Ro50} and \cite{Ro52}. We say that an isolated critical point $x_0$ satisfies the hypothesis $H$ if there is an isolating coordinate neighborhood $U$ of $x_0$ such that for all $x\ne x_0$ in $U\cap f^{-1}(x_0)$ the vectors $x-x_0$ and $\mathop\mathrm{grad} f(x)$ are linearly independent. 

\begin{corollary}\label{cor_R} Every critical point satisfying the hypothesis $H$ admits a cylindrical ball neighborhood. 
\end{corollary}

We will show that the set of jets of non cone-like critical points is of infinite codimension  in the space of all jets, see Corollary~\ref{c:t17}. In other words, non cone-like critical points are extremely rare and infinitely degenerate.  We also deduce that the critical point of any infinitely determined map germ is cone-like, see Proposition~\ref{p:inf20}. 

Finally, we show that the class of critical points that admit a cylindrical ball neighborhood diffeomorphic to a ball is larger than the class of cone-like critical points. 

\begin{theorem}\label{th:Tak} For every $m\ge 4$ there is a smooth function $f\co \R^m\to \R$ with an isolated critical point $x_0$ such that $x_0$ is not cone-like and such that $x_0$ admits a cylindrical  neighborhood diffeomorphic to a ball. 
\end{theorem}

The functions $f$ that appear in Theorem~\ref{th:Tak} are constructed by Takens in \cite{Tak}. 
Theorem~\ref{th:Tak} 
implies that critical points of Takens functions are removable, see Corollary~\ref{c:Tak26}. This 
is closely related to the Funar theorem~\cite[Proposition 2.1]{Fu21} asserting that cone-like critical points of smooth maps $f\co M\to N$ of a manifold of dimension $m\le 2k-1$ to a manifold of dimension $k\ge 2$ are topologically removable except possibly when $(m, k)\in \{(2,2), (4, 3), (8,5), (16, 9)\}$, and under an additional condition that the critical point admits an adapted neighborhood diffeomorphic to a ball when $(m,k)=(5,3)$.

Theorem~\ref{th:1} and its corollaries are related to the Lusternik-Schnirelmann category \cite{Co98, Si79, Ta68}. In fact, 
in \cite{ST}  we will show  that a closed manifold of dimension at least $6$ admits a Singhof-Takens filling by $n$ smooth balls if and only if it admits a function with $n$ critical points each of which admits a  cylindrical ball neighborhood.

In section~\ref{th:1} we give a detailed proof of Theorem~\ref{th:1}. Next, in section~\ref{s:3C4} we prove Corollary~\ref{cor_R}, and show that every cone-like critical point admits a cone neighborhood pair. In section~\ref{s:4alg} we review relevant theorems from singularity theory and deduce that the set of jets of map germs with non cone-like critical points is a subset of a proalgebraic set of infinite codimension in the space of all jets. We also show that critical points of infinitely determined map germs are cone-like. Finally, in section~\ref{s:5Tak}, we list properties of Takens critical points, namely: the Takens function germs are flat (Proposition~\ref{p:23Tak}), the critical points of Takens functions admit cylindrical neighborhoods (Proposition~\ref{prop:21}) diffeomorphic to a ball, and the critical points of Takens functions are removable (Corollary~\ref{c:Tak26}). The proof of Theorem~\ref{th:Tak} can also be found in section~\ref{s:5Tak}.

 We are grateful to Louis Funar for valuable comments, and references.

\section{Proof of Theorem~\ref{th:1}}

Let $f$ be a smooth function on a manifold $M$ without boundary. 
 We choose a Riemannian metric on $M$. It defines a gradient flow $\gamma_t$ on $M$ such that the trajectory $t\mapsto \gamma_t(y)$ of any point $y$ is a curve in $M$ parametrized by an open interval $(t_1, t_2)$ where $-\infty \le t_1 < t_2\le \infty$. Every trajectory $t\mapsto \gamma_t(y)$ is either closed in $M$ or the limit of $\gamma_t(y)$ as $t\to \infty$ or $t\to -\infty$ is a critical point of $f$. For example, for every critical point $x\in M$, its trajectory is a constant curve $t\mapsto x$ parametrized by $t\in (-\infty, \infty)$.  The union $D$ of all non-closed trajectories of the gradient flow as well as all critical points of $f$ is a closed subset of $M$.  The closed set $D$ is a union of closed subsets $D(x)$ of points $y$ such that for the trajectory $t\mapsto \gamma_t(y)$ the limit of $\gamma_t(y)$ as $t\to \infty$ or $t\to -\infty$ is the critical point $x$. 

For any value $a$ of the function $f$, let $M_a$ denote the level $f^{-1}(a)$. The intersection of $D$ with $M_a$ will be denoted by $D_a$. Similarly, $D_a(x)$ denotes the intersection of $D(x)$ and $M_a$. 

The gradient flow $\gamma_t$ defines a diffeomorphism 
\[
h_{a,b}\co M_a\setminus D_a\to M_b\setminus D_b
\]
for every pair $(a,b)$ of real numbers such that  $a <b$ by associating to a point $x$ a unique point in $M_b\setminus D_b$ on the trajectory $t\mapsto \gamma_t(x)$ of $x$. We will also write $h_{b, a}$ for the inverse of $h_{a,b}$.

Suppose now that $x_0$ is a cone-like singular point of a smooth function $f$. 
 Since $x_0$ is isolated, we may assume that $f(x_0)=0$, and $x_0$ is the only critical point on the singular hypersurface $M_0$. Similarly, we may assume that there are regular values $c$ and $-c$ of $f$ such that $0$ is a unique critical value in the interval $(-c-\epsilon, c+\epsilon)$ for some $\epsilon>0$.

Let $(B_\varepsilon, M_0\cap B_\varepsilon)$ be an isolating cone neighborhood pair for $x_0$.  We recall that $\partial B_\varepsilon$ intersects $M_0$ transversally. We will write $V$ for $M_0\cap B_\varepsilon$. 
Let $H$ denote the subset of $M_{[-c, c]}$ of points $x$ such that either the gradient curve through $x$ intersects $V$ or one of the limits of the gradient curve through $x$ is $x_0$, where $M_{[-c, c]}$ is the submanifold of $M$ of points $x$ such that $f(x)\in [-c, c]$. We write $H_a$ for the intersection of $H$ and $M_a$. By choosing $V$ and $c$ sufficiently small, we may assume that $H\subset B_\varepsilon$.

\begin{lemma}
The subsets $H_c\subset M_c$ and $H_{-c}\subset M_{-c}$ are smooth submanifolds with (smooth) boundary. 
The subset $H\subset M$ is a smooth manifold of codimension $0$ with corners $\partial H_{-c}\sqcup \partial H_{c}$. In other words, $H$ is a smooth cylindrical neighborhood of $x_0$. 
\end{lemma}

\begin{proof} Let $x$ be a point in $H\setminus D$. Suppose that $f(x)\in (-c, c)$. We will omit the cases $f(x)=c$ and $f(x)=-c$ as the argument in these cases  is similar. 
Without loss of generality we may assume that $f(x)\le 0$; the case $f(x)\ge 0$ is similar.  

Let $\delta$ denote the gradient flow of $f$. Then for some $t\ge 0$, the point $\delta_{t}(x)$ is  in $V$. Suppose $x\in H$ is a point with $\delta_t(x)\in \partial V$. Then for a small neighborhood $U$ of $x$ in $M$, there is a smooth map $\varphi\co U\to M_0$  that takes $y\in U$ to a unique point $\varphi(y)\in M_0$ on the trajectory of $y$. In view of the diffeomorphism $h_{f(x), 0}$,  the map $\varphi$ is a submersion onto a small neighborhood of $\delta_t(x)$. Therefore $x$ admits a half disc neighborhood in $H$. The case where $\delta_t(x)$ is in the interior of $V\setminus\{x_0\}$ is similar.

Suppose now that $x\in D\cap H$. We still assume that $f(x)\le 0$. We need to show that for every point $y$ sufficiently close to $x$, either $y\in D$ or the trajectory of the gradient flow of $f$ through $y$ either intersects $V$. Assume that $y$ is not in $D$, and 
the trajectory through $y$ does not intersect $V$. Then 
there is a sequence of points $y_i$ approaching $x$ such that each trajectory $t\to \delta_t(y_i)$ passes through a point $z_i$ in $M_0\setminus V$ at some moment $t_i$. We may assume $f(y_i)=f(x)$. 
Since $M_0$ is compact, there is an accumulating point $z$ of the points $z_i$ in $M_0$. In fact, by dropping some of the points $y_i$ (and $z_i$), we may assume that $z$ is a limit point of $z_i$. 
We claim that the trajectory of the gradient flow of $f$ through the point $z$ passes through the point $x$, which contradicts the assumption that the trajectory of the gradient flow through $x$ has a limit point at $x_0\in V$.  To prove the claim, let $y$ denote the point on the trajectory $t\mapsto \delta_{-t}(z)$ such that $f(y)=f(x)$. Then the trajectories through the points $z_i$ sufficiently close to $z$ intersect the level set $f^{-1}(f(y))$ at points close to $y$, i.e., the points $y_i$ approach the point $y$. Thus, $y$ coincides with $x$.

Finally, the maps $h_{-c,0}$ and $h_{0,c}$ define diffeomorphisms $\partial H_{-c}\to \partial V$ and $V\to \partial H_{c}$, and therefore, the subsets $H_c$ and $H_{-c}$ are smooth submanifolds with smooth boundary. \end{proof}

\begin{proposition}\label{prop:6} The manifold $H$ is a cone over $\partial H$ with vertex at $x_0$. 
\end{proposition}
\begin{proof}   Since $V$ is a cone over its boundary with vertex at $x_0$, it is a union of $x_0$ and $\partial V\times [0,1]$. Let $v'$ denote a vector field along the second component in $V\times [0,1]$. Let $\lambda$ denote a monotonic function on $[0,1]$ that is $0$ only at $0$ and $1$ at $1$. Let $\pi$ denote the projection $V\times [0,1]\to [0,1]$. Then $(\lambda\circ \pi)v'$ extends to a vector field $v$ on $V$. 

The vector field $v$ restricts to a tangent radial vector field over the manifold $V\setminus \{x_0\}$, it is trivial only at $0$,  and  any point in $V\setminus \{x_0\}$ escapes $V$ along $v$ in finite time. 

We will now use the gradient vector field $\mathop\mathrm{grad}(f)$  to extend the vector field $v$ over $H$, see Lemma~\ref{l:11}.  For any regular point $x\in H$ of $f$, the tangent space $T_xM$ is the direct sum of the \emph{horizontal} component $T_xM_{f(x)}$ and the \emph{vertical} component generated by $\mathop\mathrm{grad}(f)$. We say that a vector field $u$ over $H\setminus \{x_0\}$ is \emph{horizontal} if the vector $u(x)$ is in $T_xM_{f(x)}$ for all $x\in H\setminus \{x_0\}$.

\begin{lemma}\label{l:11} The vector field $v$ extends to a continuous vector field over $H$ such that $v|H\setminus \{x_0\}$ is a nowhere zero horizontal vector field.
\end{lemma}
\begin{proof} For any point $x\in H\setminus D(x_0)$ with $f(x)<0$,  
there is a unique value $t>0$ such that $\delta_t(x)$ is a point $y$ in $V$, where $\delta_t$ is the gradient flow of $f$. We define $v(x)$ by 
\[
v(x)=\pi\circ (d\delta_t)^{-1}(v(y)),
\]
 where the map $\pi$ projects the tangent space $T_yM$ at $y\in M_{[-c, c]}\setminus x_0$ to the tangent space $T_yM_{f(y)}$ of the level manifold. 
 
 We claim that $v(x)$ is non-zero. Since $\delta_t$ takes gradient curves of $f$ to gradient curves of $f$,  it takes $\nabla f(y)$  to   $\nabla f(x)$, and therefore it takes $v(y)$ to a vector linearly independent from $\nabla f(x)$, i.e., to a vector that is not in the kernel of  $\pi$. 
 
 Similarly, for $x$ in $H\setminus D(x_0)$ with $f(x)>0$ we define 
 \[
 v(x)=(\pi\circ d\delta_t)(v(y)), 
 \]
 where $x=\delta_t(y)$ for some $t>0$ and $y\in V$. Finally, we set $v(x)=0$ for 
all $x\in D(x_0)$. Then $v$ is a desired continuous vector field over $H$. 
\end{proof}

 Let $w$ denote $f\cdot \nabla f$.  Then $w(x)$ vanishes precisely over the set $M_0$ as well as over the critical points of $f$.

\begin{lemma} The vector field $v+w$ is a continuous vector field on $H$ which is nowhere zero on $H\setminus \{x_0\}$. 
\end{lemma}
\begin{proof}
Since $w$ and $v$ are continuous vector fields, we deduce that $w+v$ is a continuous vector field over $H$. It is easily verified that $w+v$ is trivial only at $x_0$. 
\end{proof}

\begin{lemma} The corners of $H$ can be smoothened so that $v+w$ is still defined on the modified manifold $H$ and $v+w$ is nowhere tangent to $\partial H$. 
\end{lemma}
\begin{proof}
Recall that the vector field $w(x)$ is a positively scaled gradient vector field over $M_{(0, c]}$ and it is a negatively scaled gradient vector field over $M_{[-c, 0)}$.  For all non-corner points $x$ of $\partial H$ on the level $\pm c$, the tangent space of $\partial H$ is tangent to the level $M_{\pm c}$, the vector $v(x)$ is also tangent to the level $M_{\pm c}$ while $w(x)$ is normal to $\partial H$ directed outward $H$. Therefore the vector field $v+w$ over $\partial H\cap M_{\pm c}$ is outward normal. 

For all non-corner points $x$ in $\partial H$ that are not on the levels $\pm c$, the vector $v(x)$ is outward normal, while $w(x)$ is tangent to $\partial H$. Again, the vector field $v+w$ is outward normal over $\partial H\setminus M_{\pm c}$. 

Finally, at any corner point $x$, locally $H$ is a quarter subspace of $\R^m$ bounded by the horizontal space $f^{-1}(c_+)$ or $f^{-1}(c_-)$,  and a vertical space composed of gradient flow curves. It follows that the corner of $H$ can be smoothened, and the vector field $v+w$ can be modified near the smoothened corner so that it is outward normal everywhere over $\partial H$, see \cite{Pu02}.  
\end{proof}

To summarize, we smoothened the corners of $H$ so that $H$ is a manifold with smooth boundary $\partial H$. We also constructed a vector field $v+w$ with a unique critical point at $x_0$. The vector field $v+w$ is outward normal everywhere over $\partial H$. 

\begin{lemma}\label{le:cylinder} The manifold $H\setminus \{x_0\}$ is diffeomorphic to $\partial H\times [0, \infty)$. 
\end{lemma}

\begin{proof}   Let $\delta^{-v-w}_t$ denote the flow along the vector field $-v-w$.  Suppose that the trajectory $t\mapsto \delta^{-v-w}_t(z)$ is well-defined for $t\in [0, a]$. Since $-v-w$ is inward normal over $\partial H$, the trajectory of any point $z$ does not escape $H$, and therefore $\delta^{-v-w}_a(z)$ is in $H$. Consequently, the curve $t\mapsto \delta^{-v-w}_t(z)$ is well-defined for $t\in [0, a+\varepsilon)$ for some $\varepsilon$. Suppose now that $t\mapsto \delta^{-v-w}_t(z)$ is well-defined for $t\in [0, a)$ and it is not well-defined for $t\in [0,a]$. Since $-v-w$ is nowhere zero, we deduce that the limit of $\delta^{-v-w}_t(z)$ as $t\to a$ is $x_0$.    

Next, we claim that for every point $x\in H\setminus \{x_0\}$, its flow curve along $-v-w$ approaches $x_0$.  Indeed, along the vector field $-w$ every point of $H$ flows towards $V$, while along the vector field $-v$ the points of $H$ stay on the same level of $f$. Consequently, any point in $H\setminus \{x_0\}$ appears arbitrarily close to $V$ after flowing along $-v-w$ for sufficiently long time. On the other hand, by the definition of $v$ any point on $V$ flows towards $x_0$ along $-v$.

Let $\lambda>0$ denote a function on $H\setminus \{x_0\}$ such that  $\lambda(x)$ and all its derivatives tend to $0$ as $x\to x_0$. We may choose the function $\lambda$ so that $\lambda\equiv 1$ near $\partial H$, and every trajectory $\delta^{-\lambda (v+w)}_t(z)$ of $-\lambda (v+w)$ starting from any point $z\in \partial V$ is well-defined for $t\in [0,\infty)$. Then the vector field $-\lambda (v+w)$ defines a diffeomorphism $\varphi\co H\setminus \{x_0\}\to \partial H\times [0, \infty)$ by $\varphi(x)=(y, t)$ for a unique point $y\in \partial H$ and a unique value $t\ge 0$ such that  $\delta_t^{-\lambda (v+w)}(y)=x$. 
\end{proof}

Thus, for every point $x\in H\setminus \{x_0\}$, the flow curve through $x$ passes through a unique point of $\partial H$, and has $x_0$ as its limit point. This implies that $H$ is a cone over $\partial H$.   
\end{proof}

\begin{corollary} The manifold $H$ is a smooth manifold homeomorphic to a ball. 
\end{corollary}
\begin{proof} Since $H$ is a cone over a manifold, and $H$ is itself a manifold, it follows that $H$ is a disc, see \cite[Lemma 12]{ST}. 
\end{proof}

Since the vector field $-v-w$ is tangent to $V\setminus \{x_0\}$ over $V\setminus \{x_0\}$, it follows that $(H, V)$ is a cone neighborhood pair for the critical points $x_0$. This completes the proof of Theorem~\ref{th:1}.

\begin{remark}
In the theory of maps with isolated cone-like critical points to manifolds of arbitrary dimension adapted neighborhoods \cite{Fu21} play an important role. In general the definition involves a list of assumptions, but in the case of a map to $\R$ the definition is equivalent to the following one. An \emph{adapted} neighborhood is a closed isolated connected neighborhood $U$ of a critical point  of a continuous function such that the level set $f^{-1}(a)$ is transverse to $\partial U$ for each $a$ in the interior of $f(U)$. Clearly, every cylindrical neighborhood is an adapted neighborhood. 
\end{remark}

\section{Proof of Corollary~\ref{cor_R}}\label{s:3C4}

Let $x_0$ be an isolated critical point of a function $f\co M\to \R$ on a smooth manifold of dimension $n$. By restricting the function $f$ to a coordinate neighborhood about $x_0$, we may assume that $M$ is $\R^n$. 
 In \cite{Ro50} E. Rothe studied the class of critical points satisfying the hypothesis $\mathcal{H}$.

\vskip 3mm
\noindent{\bf Hypothesis $\mathcal{H}$.} There exists an isolating neighborhood $U$ of $x_0$ such that for all $x\ne 0$ in $U\cap f^{-1}(x_0)$ the vectors $x-x_0$ and $\nabla f(x)$ are linearly independent.  
\vskip 3mm

The hypothesis $\mathcal{H}$ is closely related to the hypothesis $\mathcal{H}'$. 

\vskip 3mm
\noindent {\bf Hypothesis $\mathcal{H}'$.} There exists an isolating neighborhood $U$ of $x_0$ such that over the complement to $x_0$ in $V=U\cap f^{-1}(x_0)$ there is a smooth vector field $v$ which is nowhere zero and which is outward normal over $\partial V$. 

\begin{lemma} The hypothesis $\mathcal{H}$ implies the hypothesis $\mathcal{H}'$.
\end{lemma}
\begin{proof} Let $f$ be a smooth function with an isolated critical point $x_0$ that satisfies the hypothesis $\mathcal{H}$, i.e., there exists an isolating neighborhood $U$ of $x_0$ with a certain property. 
Let $r\co U\to \R$ denote the function that associates with $x$ the distance between $x$ and $x_0$ with respect to the Euclidean metric on $\R^n$. 
We may choose a ball neighborhood $U_\varepsilon$ of $x_0$ of radius $\varepsilon$ so that $U_\varepsilon\subset U$ and $\varepsilon$ is a regular value of $r$. Put $V_\varepsilon=U_\varepsilon\cap f^{-1}(x_0)$. 

Since the vectors $x-x_0$ and $\nabla f(x)$ are linearly independent, the projection $v(x)$ of the vector $x-x_0$ to $T_xV$ is non-zero for all $x\in V\setminus\{x_0\}$. Then the vector field $v$ satisfies the hypothesis $\mathcal{H}'$.
\end{proof}

The proof of the following Lemma~\ref{le:13} is similar to one of Lemma~\ref{le:cylinder}.

\begin{lemma}\label{le:13} The manifold $V\setminus \{x_0\}$ is diffeomorphic to $\partial V\times [0, \infty)$. 
\end{lemma}

We also note that the limit of $\delta^{-\lambda v}_t(z)$ as $t\to \infty$ is $x_0$ for all $z\in \partial V$. Therefore we deduce Corollary~\ref{c:8}. 

\begin{corollary}\label{c:8} The set $V$ is homeomorphic to the cone over its boundary. 
\end{corollary}

Then the proof of Proposition~\ref{prop:6} shows that a critical point satisfying the hypothesis $H$ admits a  cylindrical ball neighborhood.

The above argument proves a characterizations of a cone neighborhood pair for an isolated critical point. 
 
\begin{proposition}\label{prop:cone16} Let $B_\varepsilon$ be an $\varepsilon$-disc neighborhood of $x_0$ such that $\partial B_\varepsilon$ is transverse to $M_0$. Suppose that $V=B_\varepsilon\cap M_0$ is a cone over its boundary with vertex at $x_0$. Then there exists a cone neighborhood pair for $x_0$. 
\end{proposition}
\begin{proof} Since $V$ is a cone over its boundary with vertex at $x_0$, the complement $V\setminus \{x_0\}$ is diffeomorphic to $V\times (-\infty, 0]$. Consequently, there is a nowhere zero vector field $v$ on $V\setminus \{x_0\}$ that is outward normal over $\partial V$. By an argument as in the proof of Proposition~\ref{prop:6}, we conclude that the critical point $x_0$ admits a cylindrical ball neighborhood $H$. In fact, the construction in the proof of Proposition~\ref{prop:6} also produces a nowhere zero vector field $w+v$ over $H\setminus\{0\}$ that restricts to $v$ over $V$. Consequently, the pair $(H, H\cap V)$ is a cone neighborhood pair for $x_0$. 
\end{proof}

\section{Algebraic singular points}\label{s:4alg}

In this section we make precise the statement that the set of Taylor series of map germs with non cone-like  critical points is of infinite codimension in the space of all Taylor series. In subsection~\ref{s:tour} we  recall the Tourgeron theorem on finitely determined map germs and deduce that map germs with non cone-like critical points are extremely rare, see Corollary~\ref{c:t17}. In subsection~\ref{s:inf} we will discuss infinitely determined map germs, which is a larger class than the class of finitely determined map germs, and show that if a singular map germ is infinitely determined, then its critical point is cone-like, see Proposition~\ref{p:inf20}.

\subsection{Finitely determined map germs}\label{s:tour}
Let $\mathcal{E}(n, p)$ denote the vector space of map germs at $0$ of smooth functions  $\R^n\to \R^p$. The vector space $\mathcal{E}(n)=\mathcal{E}(n, 1)$ of function germs is a ring with multiplication and addition induced by multiplication and addition of functions. 
Let $\mathcal{B}(n)$ denote the subset of invertible map germs $f$ in $\mathcal{E}(n,n)$ with an additional condition $f(0)=0$. The set $\mathcal{B}(n)$ is a group with operation given by the composition. 
We say that map germs $f_0, f_1\in \mathcal{E}(n, p)$ are \emph{right equivalent} if there is a map germ $h\in \mathcal{B}(n)$ such that $f_0\circ h=f_1$. For a non-negative integer $k$, the \emph{$k$-jet} $\hat{f}$ of a map germ $f$ at $0$ is the Taylor polynomial of $f$ at $0$ of order $k$. Similarly, for $k=\infty$, the $k$-jet of $f$ is the Taylor series of $f$ at $0$. The vector space of all $k$-jets at $0$ is denoted by $\hat{\mathcal{E}}_k(n, p)$. The vector space of Taylor series of map germs $\R^n\to \R^p$ at $0$ is called the (infinite) \emph{jet space}. It is denoted by $\hat{\mathcal{E}}(n, p)\co =\hat{\mathcal{E}}_\infty(n, p)$. There is a sequence of projections of Euclidean spaces
\[
   \hat{\mathcal{E}}(n, p)\to \cdots \to \hat{\mathcal{E}}_{k+1}(n, p)\to \hat{\mathcal{E}}_k(n, p)\to \cdots 
\]
that truncate the series/polynomials.  The truncation $\hat{\mathcal{E}}(n, p)\to \hat{\mathcal{E}}_k(n, p)$ is denoted by $\pi_k$. We say that a subset $X$ in $\hat{\mathcal{E}}(n, p)$ is \emph{proalgebraic} if it is of the form $\cap \pi^{-1}_k X_k$ for some algebraic sets $X_k\subset \hat{\mathcal{E}}_k(n, p)$. By definition, the \emph{codimension} of $X$ is the upper limit of codimensions of the algebraic subsets $X_k$. 

A map germ $f$ is said to be \emph{$k$-definite} if every map germ $g$ with the same $k$-jet $\hat{g}=\hat{f}$ is right equivalent to $f$. For example, a function germ $f\co \R^n\to \R$ is $1$-determined at $0$ if and only if it is non-singular at $0$, i.e., $df(0)\ne 0$. A function germ $f\co \R^n\to \R$ is $2$-determined at a critical point $0$ if and only if $0$ is a Morse critical point. 
We say that a map germ $f\co (\R^n, 0)\to (\R^p, 0)$ is \emph{finite} if $\mathcal{E}(n)/\langle f_1,..., f_p\rangle$ is finite dimensional, where $\langle f_1,..., f_p\rangle$ is the ideal generated by the components $f_1,..., f_p$ of $f$. We note that the condition 
$\mathcal{E}(n)/\langle f_1,..., f_p\rangle=k$ implies \[
\dim \mathcal{E}(n)/(\langle f_1,..., f_p\rangle + m(n)^{k+1})=\dim \hat{\mathcal{E}}_k(n)/\langle \hat{f}_1,..., \hat{f}_p\rangle \le k,
\]
where $m(n)$ is the maximal ideal in $\mathcal{E}(n)$. The converse is also true, i.e., 
if we have $\dim \hat{\mathcal{E}}_k(n)/\langle \hat{f}_1,..., \hat{f}_p\rangle \le k$ for some $k$, then $f$ is finite. 
It is known that if $df\co \R^n\to \R$ is a finite map germ at $0$, then the map germ $f\co \R^n\to \R$ at $0$ is finitely determined, e.g., see \cite[\S11.10]{BL}.

On the other hand, let $Y$ denote the set of Taylor series of non-finite map germs. Then 
$
    Y=\cap \pi_k^{-1}Y_k
$
for the sets
\[
Y_k=\{\hat{g}\in \hat{\mathcal{E}}_k(n, p)\ |\  \dim \hat{\mathcal{E}}_k(n)/\langle \hat{g}_1,..., \hat{g}_p\rangle >k\},
\]
where $\hat{g}_1,..., \hat{g}_p$ are components of $\hat{g}$. 

\begin{theorem}[Tourgeron Theorem] The sets $Y_k$ are algebraic. In fact, if $n\le p$, then $Y$ is a proalgebraic set of infinite codimension. 
\end{theorem}

\begin{corollary} The set of jets of map germs $f\in \mathcal{E}(n)$ for which $df$ is not finite is a subset of infinite codimension. Thus, the set of jets of non-finitely determined map germs is a subset of a proalgebraic set of infinite codimension. 
\end{corollary}

For proofs of the Tourgeron Theorem and its corollary we refer to \cite[Theorem 13.4 and Remark 13.6]{BL}  respectively. 

We note that a $k$-determined map germ $f$ is right equivalent to a polynomial $g$ of degree $k$. Since the critical point of a polynomial $g$ is algebraic, we deduce that the critical point of any finitely determined map germ $f$ is cone-like. 

\begin{corollary} \label{c:t17} The set of jets of map germs with non cone-like critical points is a subset of a proalgebraic set of infinite codimension. 
\end{corollary}

\subsection{Infinitely determined map germs}\label{s:inf}
Next, we turn to infinitely determined map germs $f\co \R^n\to \R^p$ at $0$. 
We say that a map germ $f$ at $0$ is \emph{infinitely determined} if any map germ $g$ at $0$ that has the same Taylor series as $f$ is right equivalent to $f$.

Let $J_f$ denote the ideal in $\mathcal{E}(n)$ generated by $p\times p$ minors of the differential $d_0f$ of $f$ at $0$, where $f$ is a map germ at $0$ of a function $\R^n\to \R^p$. We say that a finitely generated ideal $I$ in $\mathcal{E}(n)$ is \emph{elliptic} if it contains $m_n^\infty$, where $m_n$ is the maximal ideal in $\mathcal{E}(n)$. Part of Theorem~\ref{th:8} was announced in \cite{Ku77} and for the proof, we refer to \cite{Br81}, \cite{N77}, and \cite{Wi81}, see also \cite[Theorem 6.1 and Lemma 6.2]{Wa81}.

\begin{theorem}\label{th:8} The following conditions are equivalent:
\begin{itemize}
\item A map germ $f$ is infinitely determined.
\item $J_f$ is elliptic.
\item $||d_xf||^2\ge C||x||^\alpha$ holds on some neighborhood of $0$ for some constants $C$ and $\alpha>0$, where  $||d_xf||^2$ is the sum of squares of $p\times p$ minors of $d_xf$. 
 \end{itemize} 
\end{theorem}

By Theorem~\ref{th:8}, if $f\co \R^n\to \R$ is not an infinitely determined function germ at $0$, then there is a sequence of points $z_1, z_2, ....$ converging to $0$ such that 	
\[
\left(\frac{\partial f}{\partial x_1}(z_i)\right)^2 + \cdots +\left(\frac{\partial f}{\partial x_n}(z_i)\right)^2=o||z_i||^N
\]
 for all $N$.

\begin{example} The map germ $f(x, y)=(x^2+y^2)^2$ is not finitely determined \cite[p. 237]{Wi82}, but it is infinitely determined. The map germ $g(x, y)=(x, y^4+xy^2)$ is not infinitely determined.  
\end{example}

\begin{proposition}\label{p:inf20} If a singular map germ is infinitely determined, then its critical point is cone-like. 
\end{proposition}
If $E$ is an equivalence relation on map germs, then we say that a map germ $f$ is $k$-$E$-determined if every map germ $g$ with the same $k$-jet $\hat{g}$ as the $k$-jet $\hat{f}$ of $f$ is $E$-equivalent to $f$. An important example of an $E$-equivalence is a $\mathcal{K}$-equivalence. If $f$ and $g$ are $\mathcal{K}$-equivalent, then there is a diffeomorphism of source spaces of $f$ and $g$ that takes $f^{-1}(0)$ to $g^{-1}(0)$.

\begin{proof} 
Suppose that $f$ is infinitely determined. Then it is infinitely $\mathcal{K}$-determined. Then, by the Brodersen theorem \cite{Br81}, it is finitely $\mathcal{K}$-determined. In particular, the map germ $f$ is $\mathcal{K}$-equivalent to a polynomial. Thus, there is a diffeomorphism of the source spaces of $f$ and $g$ that takes $f^{-1}(0)$ to an algebraic set $g^{-1}(0)$. 
\end{proof}

\section{The Takens functions}\label{s:5Tak}

In \cite{Ta68}, Takens constructed functions on manifolds of dimension $\ge 3$ with non cone-like critical points. In this section we study Takens functions, and show that the critical points of these functions still admit cylindrical neighborhoods diffeomorphic to a ball. 

Given a smooth manifold $X$ with corners and a smooth manifold $Y$ with smooth boundary, a map $f\co X\to Y$ as well as its inverse is said to be a \emph{special almost diffeomorphism} if it is continuous and restricts to a diffeomorphism $X\setminus \Sigma\to Y\setminus f(\Sigma)$ where $\Sigma$ is the set of points of the boundary of $X$ at which $\partial X$ is not smooth. An \emph{almost diffeomorphism} of manifolds with corners is a composition of special almost diffeomorphisms.

The Takens construction is based on the existence for $n\ge 4$ of a set $Q\subset \R^{n-1}\subset \R^n$ such that 

\begin{itemize}
\item there exists a continuous map $\tau\co \R^n\to \R^n$ with $\tau(Q)=0$ such that the restriction $\tau|\R^n\setminus Q$ is a diffeomorphism onto $\R^n\setminus \{0\}$, and
\item there is no closed neighborhood $U$ of $Q$ in $\R^{n-1}$ such that $U\setminus Q$ is homeomorphic with $\partial U\times (0, 1]$.  
\end{itemize}

When $n=4$, we may choose $Q\subset \R^3\subset \R^4$ to be the Fox-Artin arc, which is an embedded segment in $\R^3$ with non-trivial fundamental group $\pi_1(\R^3\setminus Q)$. When $n\ge 5$, the space $Q\subset \R^{n-1}\subset \R^n$ can be constructed by means of the Newman-Mazur compact contractible manifold $M$ of dimension $n-1$ with non-trivial $\pi_1(\partial M)$. 
The Newman-Mazur compact manifold $M$ has the properties that
\begin{itemize}
\item $M\times [0,1]$ is almost diffeomorphic to a round ball,
\item the double of $M$ is diffeomorphic to $S^n$, and \item the disc $D^{n-1}$ can be decomposed into a union $A_1\cup A_2$ of two manifolds with corners each of which is almost diffeomorphic to the Newman-Mazur manifold $M$. 
\end{itemize}
Namely, starting with a disc $B_0$ of dimension $n-1$, Takens defines a sequence of compact $(n-1)$-manifolds $B_i$, where $B_{i+1}$ is obtained from $B_i$ by removing from $B_i$ a collar  
neighborhood of $B_i$ as well as a subset diffeomorphic to $A_1$ so that $B_{i+1}$ is diffeomorphic to the boundary connected sum $B_i\#_b A_2$, which is almost diffeomorphic to $B_i\#_b M$. Then $Q=\cap B_i$. 
It follows that in all these cases $Q$ is a compact subset of $\R^{n-1}$. 

To construct a smooth function $f\co \R^n\to \R$ with an isolated non cone-like critical point $0$, we first construct a smooth function $F\co \R^n\to \R$ with $F^{-1}(0)=\R^{n-1}$ and $\partial F/\partial x_n\equiv 0$ precisely over $Q$. Then $f$ is defined by $F=f\circ \tau$. 

The function $F$ in \cite{Tak} is of the form 
\[
      F(x_1,..., x_n)=\lambda(x_1,..., x_{n-1})\cdot x_n + h(x_n), 
\]
where $\lambda(x_1,..., x_{n-1})$ and $h(x_n)$ are certain smooth functions such that $\lambda> 0$ over the complement to $Q$, $h'\ge 0$, and $h(0)=0$. Namely, let $\{Q_i\}$ be a basis of compact neighborhoods of $Q$ in $\R^{n-1}$ such that $Q_i$ is in the interior of $Q_{i-1}$. Let
$\{\lambda_i\}$ denote the partition of unity of $\R^{n-1}\setminus Q$ with support of $\lambda_1$ in the complement of $Q_2$, and support of all other functions $\lambda_i$ in $Q_{i-1}\setminus Q_{i+1}$. Then $\lambda=\sum a_i\lambda_i$, where 
the strictly positive constants $a_i$ are chosen so that 
$a_i$ is less than the minimum of $1$ and 
\[
  \frac{||x||^{2}} {2^{i}||(\lambda_i\cdot x_n)\tau^{-1}||_{C^i} }, 
\]
where $||x||^2=x_1^2+\cdots x_n^2$, and $||g||_{C^i}$ is the maximum of values of derivatives of $g$ of orders at most $i$ over the compact support of the smooth function $g$ restricted to the disc $||x||\le 1$. 
Therefore, the series $\sum (\lambda_i\cdot x_n)\circ \tau^{-1}$ converges to a smooth function $(\lambda\cdot x_n)\circ \tau^{-1}$ whose derivative of any order is obtained by term by term differentiating the series. 

Let $\{h_j\}$ denote a sequence of functions on $\R^1$ parametrized by $j=1,2, ...$ such that $h_j(x)=x$ if $|x|\ge 1/i$, $h_i(x)=0$ if $|x|<1/(i+1)$ and $h'_j(x)\ge 0$.  
As above, we define a smooth function $h(x_n)=\sum b_j h_j\circ x_n$, where the strictly positive constants $b_j$ are chosen so that $b_j$ is less than the minimum of $1$ and 
\[
   \frac{||x||^2}{2^j||h_j\circ x_n\circ \tau^{-1}||_{C^i}}. 
\]

We immediately deduce the following proposition. 

\begin{proposition}\label{p:23Tak}  The Takens function germs $f$ are flat, i.e., derivatives of all orders of $f$ at $0$ are trivial. 
\end{proposition}

\begin{proposition}\label{prop:21} The critical points of Takens functions admit cylindrical neighborhoods diffeomorphic to a ball. 
\end{proposition}
\begin{proof}
To begin with let us show that the critical points of Takens functions admit cylindrical ball neighborhoods. 

Since $F|Q\equiv 0$, and $\partial F/\partial x_n>0$ over the complement to $Q$, the implicit function theorem implies that for each $\varepsilon\ne 0$ there is a function $g_\varepsilon(x_1, ..., x_{n-1})$ such that $F(x_1,..., x_{n-1}, g_\varepsilon(x_1,..., x_{n-1}))=\varepsilon$. In other words, the level set $F(x_1,..., x_n)=\varepsilon$ is the graph of the function $x_n=g_\varepsilon(x_1,..., x_{n-1})$. 
Furthermore, let $B_r$ denote a closed ball in $\R^{n-1}$ centered at $0$ of radius $r$ such that $B_r$ contains the compact set $Q$ in its interior. Then the absolute value of the restriction $g_\varepsilon|B_r$ is bounded by a constant $c_\varepsilon$ such that $c_\varepsilon\to 0$ as $\varepsilon\to 0$. 

For any subset $X$ of $\R^{n-1}$, let $C_\varepsilon(X)$ denote the compact rectangular cylindrical subset of $\R^n$ over $X$ bounded by the graphs of the functions $g_{-\varepsilon}|X$ and $g_{\varepsilon}|X$. In other words, the set $C_\varepsilon(X)$ consists of points $(x_1,..., x_{n})$ such that $(x_1,..., x_{n-1})$ is a point in $X$ and 
\[
   g_{-\varepsilon}(x_1,..., x_{n-1})\le x_n \le g_{\varepsilon}(x_1,..., x_n). 
\] 
Then $C_\varepsilon(X)$ is diffeomorphic to the cylinder $X\times [-\varepsilon, \varepsilon]$.   

Let $H(r, \varepsilon)$ denote a neighborhood of $Q$ that consists of all trajectories of the gradient vector field of $F$ in $F^{-1}[-\varepsilon, \varepsilon]$ that either contain a point in $B_r$ or has a limit point in $Q$.

\begin{lemma} There is a number $r_3>0$ such that the cylindrical neighborhood $H(r_3, \varepsilon)$ is diffeomorphic to the rectangular cylinder $C_\varepsilon(B_{r_3})$. 
\end{lemma}
\begin{proof}
Let $r_0<r_1 < r_2<r_3<r_4$ be numbers such that $B_{r_0}$ contains $Q$ in its interior, and 
$H(r_i, \varepsilon)\subset C_\varepsilon(B_{r_{i+1}})$ for $i=0,..., 3$. 
There is a smooth homotopy $v_t$ parametrized by $t\in [0,1]$ of the gradient vector field $v_0$ of $F$ to a vector field $v_1$ such that $v_t$ is constant over $H(r_0, \varepsilon)$, $v_t\cdot \partial/\partial x_n>0$ over $C_\varepsilon(B_{r_3})$ for all $t$, and $v_1$ is vertical up over $C_{\varepsilon}(B_{r_4}\setminus B_{r_2})$. 

Given a point $x$ in $\R^{n-1}\setminus Q$, there is a unique trajectory $\gamma_x$ of the vector filed $v_t$ passing through $x$. Let $y\in \R^n$ denote the point in $\R^n$ on the trajectory $\gamma_x$ such that $F(y)=s$. Then $j_{t, s}\co \R^{n-1}\setminus Q\to \R^n$ defined by $x\mapsto y$ is a smooth map that takes $B_{r_3}\setminus Q$ to the upper horizontal part $\{F\equiv s\}\cap H(r_3, s)$ of the boundary of $H(r_3, s)$.

Then there is  a smooth isotopy $\varphi_t$ of a neighborhood $H(r_3, \varepsilon)$ of $Q$ to  $C_\varepsilon(B_{r_3})$ defined by 
\[
\varphi_t(x_1,..., x_n) = j_{t, F(x_n)} (j_{0, F(x_n)}^{-1}(x_1,..., x_n)).
\]
We claim that $\varphi_1$ is a diffeomorphism onto $C_\varepsilon(B_{r_3})$. Indeed, let $(x_1,..., x_{n})$ be a point in $C_\varepsilon(B_{r_3})$ that is not on the trajectories of $v_1$ with limit points on $Q$. 
Then the point $(x_1,..., x_{n})$  is the image of the point 
 \[
     j_{0, F(x_n)}\circ j_{1, F(x_n)}^{-1}(x_1,..., x_n).
 \]
On the other hand, if $(x_1,..., x_n)$ is on a trajectory of $v_1$ with limit point on $Q$, then it is in $H(r_0, \varepsilon)$ where the isotopy $v_t$ is constant. Therefore, the point $(x_1,..., x_n)$ in this case is the image of itself.
\end{proof}  

Since $H=H(r_3, \varepsilon)$ is diffeomorphic to $C_\varepsilon(B_{r_3})$, we conclude that $H$ is a ball with corners $B_{r_3}\times [-\varepsilon, \varepsilon]$. Consequently, $H/Q$ is a cylindrical ball neighborhood of the critical point of $f$. 

 Now let us show that the cylindrical neighborhood $H/Q$ is not only homeomorphic to a ball, but it is almost diffeomorphic to a ball. Indeed, if $n\ne 4, 5$, then every smooth manifold homeomorphic to a ball is diffeomorphic to ball. Suppose that $n=5$.  Since the boundary of $H/Q$ is diffeomorphic to the boundary of $H$, which is almost diffeomorphic to a sphere, we conclude that $H/Q$ is almost diffeomorphic to a ball, see \cite[Proposition C]{Mi65}. Suppose now that $n=4$. Then there is an ambient  isotopy of $H$ that is trivial near the boundary, and that takes the wild arc $Q$ into the unknotted standard segment $I$. Consequently, the manifold $H/Q$ is diffeomorphic to the manifold $H/I$, which is almost diffeomorphic to a smooth $4$-ball. 
\end{proof}

\begin{proof}[Proof of Theorem~\ref{th:Tak}] It is known that the critical point of a Takens function is not cone-like. On the other hand, by Proposition~\ref{prop:21}, such a critical point admits a cylindrical neighborhood diffeomorphic to a ball. 
\end{proof} 

\begin{corollary}\label{c:Tak26}  The critical points of Takens functions are removable, i.e., for any neighborhood $U$ of a critical point $x_0$ of a Takens function $f$, there is a smooth function $g$ such that $g\equiv f$ outside of $U$ and $g$ has no critical points in $U$. 
\end{corollary}
\begin{proof} In the proof of Proposition~\ref{prop:21} we may choose $\varepsilon, r_0, r_1, r_2, r_3,$ and $r_4$ so small that $H/Q\subset U$. We note that the boundary of the ball $H/Q$ is diffeomorphic to the boundary of the ball $H$, and the function $f$ near $\partial (H/Q)$ can be identified with the function $F$ near $\partial H$. On the other hand, the boundary $\partial H$ consists of the upper and lower horizontal parts $\partial_hH=\partial H\cap \{F=\pm \varepsilon\}$ and the vertical part $\partial_vH=\overline{\partial H\setminus \partial_hH}$ such that $F$ is constant over each component of $\partial_hH$ and increasing over $\partial_vH$. Therefore, there is a function $f'$ on $H/Q$ such that $f'$ agrees with $f$ near $\partial (H/Q)=\partial H$ and $f'$ has no critical points. Finally, we define $g$ to be a function on $\R^{n}$ such that $g\equiv f'$ over $H/Q$ and $g\equiv f$ otherwise.  
\end{proof}

\end{document}